\documentclass{article}
\usepackage{amsmath,amssymb,amsthm,fullpage,url}

\theoremstyle{plain}
\newtheorem{proposition}{Proposition}
\newtheorem{lemma}{Lemma}

\title{A Note on Centralizers of Involutions in Coxeter Groups}
\author{Koji Nuida\\ Institute of Mathematics for Industry (IMI), Kyushu University\\ \url{nuida@imi.kyushu-u.ac.jp}}
\date{\vspace*{-2em}}
\begin{document}
\maketitle

\begin{abstract}
    In this note, we give a remark on the structure of centralizers of involutions in Coxeter groups. 
\end{abstract}

\section{Introduction}

Centralizers and normalizers of some subgroups are one of the major objects in studying structural properties of groups.
Here we use the following notation for centralizers
\[
    Z_G(H) := \{ g \in G \mid gh = hg \mbox{ for any } h \in H \} \enspace,
\]
and the following notation for normalizers
\[
    N_G(H) := \{ g \in G \mid gHg^{-1} \subseteq H \} \enspace.
\]
For the case of Coxeter groups, centralizers of reflections are described by Brink \cite{Brink96} and also studied by Allcock \cite{Allcock13}.
Centralizers of Coxeter elements are studied by Blokhina \cite{Blokhina89}, by Kaul and White \cite{KauWhi09}, and by Hollenbach and Wegener \cite{HolWeg22}.
Centralizers of general elements in finite Coxeter groups are studied by Konvalinka, Pfeiffer, and R\"{o}ver \cite{KPR11}.
Normalizers of parabolic subgroups are described by Howlett \cite{Howlett80} for finite Coxeter groups and by Brink and Howlett \cite{BriHow99} for a general case (see also \cite{Allcock12,Borcherds98}).
Centralizers of parabolic subgroups are described by Nuida \cite{Nuida11,Nuida13}.

In a recent preprint, Serre \cite{Serre22} described centralizers of involutions in finite Coxeter groups.
The aim of this note is to give another way of describing centralizers of involutions in Coxeter groups by using some known structural results on Coxeter groups, such as the classification of conjugacy classes of involutions in Coxeter groups \cite{Richardson82}.

\section{Results}

The definitions and properties for Coxeter groups used in this note without mentioning can be found in the book \cite{Humphreys}.
For a Coxeter system $(W,S)$, let $V$ be the standard geometric representation space spanned by the set of simple roots $\Pi = \{ \alpha_s \mid s \in S \}$.
Then the root system $\Phi = W \cdot \Pi$ is the disjoint union $\Phi = \Phi^+ \sqcup \Phi^-$ where $\Phi^+$ and $\Phi^-$ are the sets of positive and negative roots, respectively.
For a subset $I \subseteq S$, let $W_I := \langle I \rangle$ denote the standard parabolic subgroup of $W$ generated by $I$.
Then $(W_I,I)$ is also a Coxeter system whose standard geometric representation space is the subspace $V_I$ of $V$ spanned by the elements $\alpha_s$ with $s \in I$.
The set $\Phi_I := \Phi \cap V_I$ forms the root system of $(W_I,I)$ with the decomposition into $\Phi^+_I := \Phi_I \cap \Phi^+$ and $\Phi^-_I := \Phi_I \cap \Phi^-$.
For $w \in W$, let $\ell(w)$ denote the length of the shortest expression of $w$ as a product of elements of $S$.
If $I \subseteq S$ and $|W_I| < \infty$, then $W_I$ has the unique element of the maximum length, called the longest element of $(W_I,I)$ and denoted here by $\rho_I$.

We say that a Coxeter system $(W,S)$ is of \emph{$(-1)$-type} if $W$ is finite and the action of the longest element $\rho = \rho_S$ of $(W,S)$ satisfies that $\rho \cdot \alpha_s = -\alpha_s$ for any $s \in S$.
The following is a well-known result by Richardson \cite[Theorem A]{Richardson82} on the conjugacy classes for involutions in Coxeter groups.

\begin{proposition}
    \label{prop:Richardson}
    Let $(W,S)$ be a Coxeter system and $w$ be an element of $W$ with $w^2 = 1$.
    Then there exist a subset $I \subseteq S$ and an element $u \in W$ satisfying that $(W_I,I)$ is of $(-1)$-type and $uwu^{-1} = \rho_I$.
\end{proposition}

The following result was proved by Felder and Veselov \cite[Proposition 7]{FelVes05} and by Pfeiffer and R\"{o}hrle \cite[Proposition 2.2]{PfeRoh05} for the case of finite Coxeter groups and by M\"{u}hlherr and Nuida \cite[Proposition 2.16 (i)]{MuhNui16} for a general case.

\begin{proposition}
    \label{prop:centralizer_and_normalizer}
    Let $(W,S)$ be a Coxeter system and $I$ be a subset of $S$ satisfying that $(W_I,I)$ is of $(-1)$-type.
    Then $Z_W(\rho_I) = N_W(W_I)$.
\end{proposition}

By combining Propositions \ref{prop:Richardson} and \ref{prop:centralizer_and_normalizer}, studying the structure of the centralizer $Z_W(w)$ for an element $w$ with $w^2 = 1$ is reduced to finding a pair of a subset $I \subseteq S$ and an element $u \in W$ satisfying that $(W_I,I)$ is of $(-1)$-type and $uwu^{-1} = \rho_I$ and studying the structure of the normalizer $N_W(W_I)$, as now we have $Z_W(w) = u^{-1} Z_W(\rho_I) u = u^{-1} N_W(W_I) u$.
The second part of the latter task is just an application of the result of \cite{Howlett80} when $W$ is finite and of \cite{BriHow99} for a general case.

We describe an algorithm for the first part of the latter task above.
Let $(W,S)$ be a Coxeter system and let $w \in W$.
We define
\[
    \Phi[w] := \{ \gamma \in \Phi^+ \mid w \cdot \gamma \in \Phi^- \} \enspace,
\]
\[
    J_w := \{ s \in S \mid w \cdot \alpha_s \in \Phi^- \} \enspace,
\]
\[
    K_w := \{ s \in S \mid w \cdot \alpha_s = -\alpha_s \} \enspace.
\]
Note that $K_w \subseteq J_w$ by definition, and we have $\Phi^+_{J_w} \subseteq \Phi[w]$, which is finite (as $|\Phi[w]| = \ell(w) < \infty$), therefore $|W_{J_w}| < \infty$.
Now we have the following two results:

\begin{lemma}
    \label{lem:when_involution_is_rho}
    In the situation above, if $J_w = K_w$, then $(W_{K_w},K_w)$ is of $(-1)$-type and $w = \rho_{K_w}$.
\end{lemma}
\begin{proof}
    First we show that the claim will follow once it is shown that $\Phi[w] \subseteq \Phi_{K_w}$.
    Indeed, $w$ admits a decomposition $w = w^{K_w} w_{K_w}$ into an element $w_{K_w} \in W_{K_w}$ and the shortest representative $w^{K_w}$ of the coset $w W_{K_w}$, and it satisfies that $\Phi[w^{K_w}] \cap \Phi_{K_w} = \emptyset$ and hence $\Phi[w_{K_w}] = \Phi[w] \cap \Phi_{K_W}$.
    Now if $\Phi[w] \subseteq \Phi_{K_w}$, then we have $\Phi[w_{K_w}] = \Phi[w]$ which implies that $\Phi[w_{K_w}w^{-1}] = \emptyset$ and hence $w_{K_w}w^{-1} = 1$ and $w_{K_w} = w$.
    Therefore we have $w_{K_w} \cdot \alpha_s = w \cdot \alpha_s = -\alpha_s$ for any $s \in K_w$ and hence $(W_{K_w},K_w)$ is of $(-1)$-type and $w = w_{K_w} = \rho_{K_w}$, as desired.

    Our remaining task is to show that $\Phi[w] \subseteq \Phi_{K_w}$.
    Assume for the contrary that $\gamma \in \Phi^+ \setminus \Phi_{K_w}$ and $w \cdot \gamma \in \Phi^-$.
    Write $\gamma = \sum_{s \in S} c_s \alpha_s$ with $c_s \geq 0$.
    Then we have
    \[
        w \cdot \gamma = \sum_{s \in K_w} (-c_s) \alpha_s + \sum_{s \in S \setminus K_w\,;\,c_s > 0} c_s w \cdot \alpha_s \in \Phi^- \enspace.
    \]
    Now the second sum in the right-hand side has at least one term (by the assumption $\gamma \not\in \Phi_{K_w}$) and each of such terms satisfies that $w \cdot \alpha_s \in \Phi^+$ (as $K_w = J_w$).
    This implies that, for each of such terms, we must have $w \cdot \alpha_s \in \Phi_{K_w}$ and hence $\alpha_s = w^{-1} \cdot (w \cdot \alpha_s) = -w \cdot \alpha_s \in \Phi^-$ by the definition of $K_w$.
    This is a contradiction.
    Hence we have $\Phi[w] \subseteq \Phi_{K_w}$, concluding the proof.
\end{proof}

\begin{lemma}
    \label{lem:shorten_involution}
    In the situation above, if $w^2 = 1$ and $s \in J_w \setminus K_w$, then $\ell(sws) < \ell(w)$.
\end{lemma}
\begin{proof}
    Assume for the contrary that $\ell(sws) \geq \ell(w)$.
    The assumption $s \in J_w$ implies that $\ell(ws) < \ell(w)$, therefore $\ell(sws) > \ell(ws) = \ell((ws)^{-1}) = \ell(sw)$ as $w^2 = 1$.
    These inequalities for lengths imply that $w \cdot \alpha_s \in \Phi^-$ and $sw \cdot \alpha_s \in \Phi^+$, therefore $w \cdot \alpha_s = -\alpha_s$.
    This contradicts the assumption $s \not\in K_w$.
    Hence the claim holds.
\end{proof}

By combining Lemmas \ref{lem:when_involution_is_rho} and \ref{lem:shorten_involution}, we can find, for any given element $w \in W$ with $w^2 = 1$, a subset $I \subseteq S$ and an element $u \in W$ satisfying that $(W_I,I)$ is of $(-1)$-type and $uwu^{-1} = \rho_I$.
Indeed, if $J_w = K_w$ (which also holds in the base case $w = 1$), then $I := K_w$ and $u := 1$ satisfy the condition by Lemma \ref{lem:when_involution_is_rho}.
Otherwise, we have $J_w \setminus K_w \neq \emptyset$, and any of $s \in J_w \setminus K_w$ satisfies that $\ell(sws) < \ell(w)$ by Lemma \ref{lem:shorten_involution}.
Now the recursive procedure for the shorter involution $sws$ yields such objects $I' \subseteq S$ and $u' \in W$ for $sws$; then $I := I'$ and $u := u's$ satisfy the condition for $w$.

\end{document}